\documentclass[11pt,a4paper]{article}
\usepackage[utf8]{inputenc}

\usepackage{graphicx}
\usepackage{amsmath}
\usepackage[margin=2.5cm]{geometry}
\usepackage{amsmath,amssymb,amsthm}
\usepackage{booktabs}
\usepackage[hidelinks]{hyperref}
\newtheorem{theorem}{Theorem}
\newtheorem{lemma}[theorem]{Lemma}
\newtheorem{corollary}[theorem]{Corollary}
\newtheorem{proposition}[theorem]{Proposition}
\theoremstyle{definition}
\newtheorem{definition}[theorem]{Definition}
\newtheorem{problem}[theorem]{Problem}
\newtheorem{protocol}[theorem]{Protocol}
\newtheorem{example}[theorem]{Example}
\newtheorem{remark}[theorem]{Remark}

\newcommand{\N}{\mathbb{N}}
\newcommand{\Circ}{\operatorname{Circ}}
\newcommand{\Exp}{\operatorname{Exp}}
\newcommand{\Mat}{\operatorname{Mat}}
\newcommand{\ord}{\operatorname{ord}}
\newcommand{\per}{\operatorname{per}}
\newcommand{\pr}{\operatorname{pr}}
\newcommand{\lcm}{\operatorname{lcm}}
\newcommand{\CR}{C_R}

\newcommand{\smallmat}[1]{{\setlength{\arraycolsep}{2.2pt}\renewcommand{\arraystretch}{0.95}\left(\begin{array}{*{25}{c}}#1\end{array}\right)}}

\title{A Key-Exchange Protocol Based on the Action of Circulant Matrices over Congruence-Simple Semirings}
\author{Álvaro Otero Sánchez\thanks{Corresponding author: \texttt{aos073@ual.es}}\\[2pt]
\small Department of Mathematics, University of Almería, 04120 Almería, Spain}
\date{September 18, 2026}

\begin{document}
\maketitle

\begin{abstract}
We present a key-exchange protocol based on the action of circulant matrices over $\N$ on tuples of pairwise commuting matrices over a congruence-simple semiring. The proposed exponentiation action replaces the polynomial action used in earlier semiring-based constructions and thereby avoids the linear structure exploited by the known polynomial attacks. We give explicit matrix constructions, estimate the computational cost of the protocol, and discuss its security against known classical and quantum attacks.
\end{abstract}

\section{Introduction}

The paper of Diffie and Hellman \cite{DH} is usually regarded as the beginning of modern public-key cryptography. In its original formulation, the key-exchange protocol works in the multiplicative group of a finite prime field. Later, Koblitz \cite{Koblitz} and Miller \cite{Miller} independently proposed an analogous protocol based on elliptic curves.

These protocols have become the standard in cryptography, and their use is widely spread. However, in \cite{Shor}, Shor presented a quantum algorithm capable of solving the discrete logarithm problem in both algebraic settings. Since the security of these cryptographic protocols is based on the hardness of this problem, a sufficiently large quantum computer would be able to break most of today’s key exchange protocols. As a result, the scientific community has initiated the search for new post-quantum cryptographic protocols, that is, protocols secure against quantum attacks.

One such proposal was made by Maze, Monico and Rosenthal \cite{MMR}, who introduced key exchange based on actions of abelian semigroups. The security of their protocol relies on the following problem.

\begin{problem}[Semigroup Action Problem]
Let $(G,\cdot)$ be a semigroup, let $S$ be a set, and let $\varphi\colon G\times S\to S$ be an action of $G$ on $S$, that is, $\varphi(g\cdot h,s)=\varphi(g,\varphi(h,s))$ for all $g,h\in G$ and $s\in S$. Given $x\in S$ and $y\in\varphi(G,x)$, find $g\in G$ such that $\varphi(g,x)=y$.
\end{problem}

In \cite{MMR}, the authors define a general framework for a new key exchange protocol using semimodules, which are analogous to modules but where the acting ring is replaced by a semiring, and the underlying structure is a commutative monoid. As an explicit example, they use congruence-simple semirings. In this setting, polynomials over the center of the semiring act on matrices over the semiring via evaluation at two fixed matrices and multiplication. They present an example using a semiring with six elements.

Steinwandt and Suárez \cite{SS} cryptanalysed this specific example by solving a system of equations derived from the operation tables of the semiring. The security of the protocol for other semirings remained open until \cite{OL}, where a general attack was developed for every congruence-simple semiring, assuming only a bound on the degree of the polynomials used in the protocol.

The aim of this paper is to introduce a different semigroup action in the setting of matrices over congruence-simple semirings. Instead of polynomial evaluation, the acting operation is exponentiation, and the exponents are organised into circulant matrices. Let $T$ be a set of pairwise commuting elements of a monoid and let $v=(v_0,\dots,v_{n-1})\in T^n$. For a circulant matrix $A=(a_{ij})\in\Circ_n(\N)$ we define
\[
\Exp(A,v)=\Bigl(\prod_{j=0}^{n-1}v_j^{a_{ij}}\Bigr)_{i=0}^{n-1}.
\]
Equivalently, if $A=\Circ(a_0,\dots,a_{n-1})$, then the $i$th component of $\Exp(A,v)$ is $\prod_{j=0}^{n-1}v_j^{a_{i-j}}$, where indices are taken modulo~$n$.

The choice of circulant matrices is essential. The set $\Circ_n(\N)$ is a commutative semigroup under matrix multiplication and, since the entries of $v$ commute, the exponentiation satisfies $\Exp(A,\Exp(B,v))=\Exp(AB,v)$ for all $A,B\in\Circ_n(\N)$. In particular, $\Exp(A,\Exp(B,v))=\Exp(B,\Exp(A,v))$, and hence the action gives rise to a key-exchange protocol.

More concretely, Alice and Bob publicly agree on a congruence-simple semiring $R$, a matrix $M\in\Mat_m(R)$ and a tuple $v=(v_0,\dots,v_{n-1})\in\CR[M]^n$, where $\CR$ denotes the centre of $R$ and $\CR[M]$ is the commutative subsemiring generated by $M$ and $\CR$. Alice chooses a private circulant matrix $A\in\Circ_n(\N)$ and publishes $\Exp(A,v)$. Similarly, Bob chooses a private circulant matrix $B\in\Circ_n(\N)$ and publishes $\Exp(B,v)$. Both parties then compute
\[
\Exp(A,\Exp(B,v))=\Exp(AB,v)=\Exp(BA,v)=\Exp(B,\Exp(A,v)),
\]
which is the shared secret.

An action of this type is not new. Maze, Monico and Rosenthal \cite{MMR02} first
considered the action of $\N[C]$, for a companion matrix $C$, on direct products
$G^n$ of a group $G$, and Huang and Li \cite{HL22} recently used the corresponding
action of $\N[C]$ on tuples of pairwise commuting tropical matrices to build a
key-exchange protocol based on what they call the multiple exponentiation
problem. For $f(x)=x^n-1$, the choice recommended in \cite{HL22}, the semiring
$\N[C]$ is precisely $\Circ_n(\N)$, so the action considered here coincides with
theirs. The present work was developed independently, originating in the first
author's undergraduate thesis, and the author became aware
of \cite{HL22} only after the present study had been completed. Our contribution
is therefore not the action itself, but its study over finite congruence-simple
semirings: the construction of matrices of large order via block-cyclic matrices
and conjugation by generalised permutation matrices, a sufficient condition for
the additive irreducibility of a finite semigroup that justifies this
construction, uniqueness results for the private key, and an analysis of why the
attack of \cite{OL} on the polynomial action of \cite{MMR} does not carry over to
the exponentiation setting.

The proposed construction differs substantially from the polynomial action of \cite{MMR}. In the earlier setting, the secret information is encoded by polynomials, and the public action arises from evaluating these polynomials at fixed matrices. In the present setting, the private information is encoded by a circulant matrix of non-negative integer exponents, and the action is given by multiplicative combinations of powers of pairwise commuting matrices. This distinction is important from a cryptanalytic point of view. The attack in \cite{OL} exploits the fact that, in the polynomial representation of the original action, the unknowns appear linearly and every solution of the resulting linear system is automatically compatible with the other party's action. As we show in Section \ref{sec:OL}, neither property survives in the exponentiation action: an expression that reproduces the action on the public tuple need not reproduce it on a tuple already transformed by another circulant matrix.

A further issue is the construction of suitable public matrices. The matrix $M$ should have many distinct powers, so that the public tuple provides a large search space, while the matrices $v_i$ must lie in a commutative subset for the action to be well defined. We address this problem using the commutative subsemiring $\CR[M]$ and constructions based on block-cyclic matrices. Starting from a partition of the matrix dimension with a large least common multiple, we obtain a matrix whose order is controlled by the orders of cyclic permutation blocks. We then modify the entries above the block diagonal and conjugate by a generalised permutation matrix. Conjugation preserves the order of the matrix while concealing the initial block structure.

We also study the additive structure of the underlying semiring. We give a sufficient condition for a finite additive semigroup to be irreducible. This result applies to the semirings considered in this work and, together with known results on invertibility over finite additively irreducible semirings, implies that the invertible matrices are exactly the generalised permutation matrices. This fact justifies the use of such matrices in the conjugation step of the construction.

Finally, we analyse the computational cost of the protocol, compare it with the protocol of \cite{MMR} and with classical and standardised post-quantum schemes, and discuss concrete parameters. We consider brute-force and random attacks and give conditions under which the private circulant matrix is uniquely determined by its action on the public tuple. We also examine the possible applicability of Pohlig--Hellman-type reductions and of the attack of \cite{OL} on the original polynomial construction. Finally, we discuss the quantum setting. If the commutative tuple consists only of powers of one matrix, quantum discrete-logarithm algorithms lead to systems of linear Diophantine equations. This observation motivates the use of more general elements of $\CR[M]$ rather than powers of a single matrix.

The paper is organised as follows. In Section \ref{sec:math} we introduce the algebraic background, define the action of circulant matrices on tuples of commuting matrices, present the key-exchange protocol, describe the construction of matrices of large order and give explicit examples. In Section \ref{sec:security} we study the computational cost of the protocol, suggest parameters and analyse its resistance to brute-force, random, Pohlig--Hellman-type, O--L-type and quantum attacks. Section \ref{sec:conclusion} contains concluding remarks and directions for further research.

\section{Mathematical setting}\label{sec:math}

We first recall the notion of a semigroup.

\begin{definition}
A \emph{semigroup} is a set $S$ equipped with an associative binary operation $\cdot\colon S\times S\to S$, that is, $(a\cdot b)\cdot c=a\cdot(b\cdot c)$ for all $a,b,c\in S$. The semigroup $S$ is \emph{commutative} if, in addition, $a\cdot b=b\cdot a$ for all $a,b\in S$. A semigroup with an identity element is a \emph{monoid}.
\end{definition}

Next, we recall the well-known notion of a circulant matrix.

\begin{definition}
Let $R$ be a semiring. A matrix $C\in\Mat_n(R)$ is \emph{circulant} if there are $c_0,c_1,\dots,c_{n-1}\in R$ such that
\[
C=\begin{pmatrix}
c_0 & c_{n-1} & c_{n-2} & \cdots & c_1\\
c_1 & c_0 & c_{n-1} & \cdots & c_2\\
c_2 & c_1 & c_0 & \cdots & c_3\\
\vdots & \vdots & \vdots & \ddots & \vdots\\
c_{n-1} & c_{n-2} & c_{n-3} & \cdots & c_0
\end{pmatrix},
\]
that is, $C_{ij}=c_{i-j}$ with indices taken modulo $n$. We write $C=\Circ(c_0,\dots,c_{n-1})$ and denote by $\Circ_n(R)$ the set of $n\times n$ circulant matrices over $R$.
\end{definition}

A classical result on the structure of circulant matrices is the following.

\begin{theorem}
Let $R$ be a commutative semiring. The set $\Circ_n(R)$ of $n\times n$ circulant matrices over $R$ is a commutative subsemiring of $\Mat_n(R)$. In particular, $\Circ_n(\N)$ is a commutative semigroup under matrix multiplication.
\end{theorem}

This allows us to define a new action.

\begin{theorem}\label{thm:action}
Let $S$ be a monoid and let $T\subseteq S$ be a set of pairwise commuting elements, that is,
\begin{equation}
ab=ba\quad\text{for all } a,b\in T.
\end{equation}
Let $\langle T\rangle$ be the (commutative) submonoid generated by $T$. Then the map
\[
\Exp\colon \Circ_n(\N)\times\langle T\rangle^n\longrightarrow\langle T\rangle^n,\qquad
\bigl(C,(v_j)_{j=0}^{n-1}\bigr)\longmapsto\Bigl(\prod_{j=0}^{n-1}v_j^{c_{i-j}}\Bigr)_{i=0}^{n-1},
\]
where $C=\Circ(c_0,\dots,c_{n-1})$, is an action of the multiplicative semigroup $\Circ_n(\N)$ on $\langle T\rangle^n$. We also write $Cv$ for $\Exp(C,v)$.
\end{theorem}

\begin{proof}
Since $\langle T\rangle$ is a commutative submonoid, $\Exp(C,v)\in\langle T\rangle^n$ for all $C\in\Circ_n(\N)$ and $v\in\langle T\rangle^n$, and the order of the factors in each product is irrelevant. Let $A,B\in\Circ_n(\N)$ and write $A_{il}$, $B_{lj}$ for their entries, so that $\Exp(A,v)_i=\prod_j v_j^{A_{ij}}$. Then, for every $i$,
\begin{align*}
\Exp\bigl(A,\Exp(B,v)\bigr)_i
&=\prod_{l=0}^{n-1}\Bigl(\prod_{j=0}^{n-1}v_j^{B_{lj}}\Bigr)^{A_{il}}
=\prod_{l=0}^{n-1}\prod_{j=0}^{n-1}v_j^{A_{il}B_{lj}}\\
&=\prod_{j=0}^{n-1}v_j^{\sum_{l}A_{il}B_{lj}}
=\prod_{j=0}^{n-1}v_j^{(AB)_{ij}}
=\Exp(AB,v)_i,
\end{align*}
where we have used that the $v_j$ commute. Since $AB$ is again circulant, this proves the claim.
\end{proof}

\begin{remark}
The computation above holds for arbitrary matrices $A,B\in\Mat_n(\N)$. The circulant structure is needed only to guarantee that $AB=BA$, which is what makes the key exchange work. The action of Theorem \ref{thm:action} is a particular instance of the framework of \cite{MMR}: it suffices to consider the $\N$-semimodule generated by $T$ with the $\N$-action $k\cdot t=t^k$.
\end{remark}

As the set $T$ we use a commutative set of matrices over a semiring.

\begin{definition}
A \emph{semiring} is a set $R$ with two associative binary operations $+$ and $\cdot$ such that $a\cdot(b+c)=a\cdot b+a\cdot c$ and $(b+c)\cdot a=b\cdot a+c\cdot a$ for all $a,b,c\in R$. The semiring $R$ is \emph{additively} (respectively \emph{multiplicatively}) \emph{commutative} if its addition (respectively multiplication) is commutative, and it is \emph{commutative} if both operations are commutative.
\end{definition}

The study of semirings for cryptographic purposes was initiated in \cite{MMR}, where the authors focus on congruence-simple semirings.

\begin{definition}
A \emph{congruence} on a semiring $R$ is an equivalence relation $\sim$ such that, for all $a,b,c\in R$,
\[
a\sim b\ \Longrightarrow\ a+c\sim b+c,\quad c+a\sim c+b,\quad ac\sim bc,\quad ca\sim cb.
\]
A semiring $R$ is \emph{congruence-simple} if its only congruences are the trivial ones, $\mathrm{id}_R$ and $R\times R$.
\end{definition}

The following classification of finite congruence-simple semirings first appeared in \cite{Monico}.

\begin{theorem}[{\cite[Theorem~4.1]{Monico}}]
Let $R$ be a finite, additively commutative, congruence-simple semiring. Then one of the following holds:
\begin{enumerate}
\item $|R|\le 2$;
\item $R\cong\Mat_n(\mathbb{F}_q)$ for some finite field $\mathbb{F}_q$ and some $n\ge1$;
\item $R$ is a zero-multiplication ring of prime order;
\item $R$ is additively idempotent, that is, $x+x=x$ for every $x\in R$;
\item $R$ contains an absorbing element $\infty$, that is, $x\cdot\infty=\infty\cdot x=\infty$ for every $x\in R$, and $R+R=\{\infty\}$.
\end{enumerate}
\end{theorem}

To obtain a commutative set, we use the following construction from \cite{Monico}.

\begin{definition}
Let $R$ be a semiring with centre $\CR=\{r\in R: rs=sr \text{ for all } s\in R\}$ and let $A\in\Mat_n(R)$. We denote by $\CR[A]$ the set of polynomials in $A$ with coefficients in $\CR$.
\end{definition}

\begin{lemma}
Let $R$, $\CR$ and $A$ be as above. Then $\CR[A]$ is a commutative subsemiring of $\Mat_n(R)$.
\end{lemma}

As a result, we will use the multiplicative semigroup of matrices over a semiring, and as a commutative subset, we will use elements of \( C_R[A] \), with $A\in Mat_n(R)$ and $n\in \mathbb{N}$. In addition, in \cite{Monico} it is also proven that

\begin{lemma}
Let $R$ be an additively commutative semiring with $0$ and $1$, and let $\sim$ be a congruence on $\Mat_n(R)$. Then there exists a congruence $\sim_0$ on $R$ such that, for all $A,B\in\Mat_n(R)$,
\[
A\sim B\iff a_{ij}\sim_0 b_{ij}\quad\text{for all } 1\le i,j\le n.
\]
\end{lemma}

In particular, $\Mat_n(R)$ is congruence-simple whenever $R$ is. Under these assumptions, we propose the following key-exchange protocol. Throughout, $m$ denotes the size of the matrices over $R$, $n$ the length of the public tuple and $k$ an upper bound for the entries of the private circulant matrices.

\begin{protocol}\label{prot}
Alice and Bob agree on a congruence-simple semiring $R$, a matrix $M\in\Mat_m(R)$ of large order and a tuple $v=(v_0,\dots,v_{n-1})\in\CR[M]^n$, and make them public.
\begin{enumerate}
\item Alice chooses $A\in\Circ_n(\N)$ and publishes $pk_1=\Exp(A,v)$.
\item Bob chooses $B\in\Circ_n(\N)$ and publishes $pk_2=\Exp(B,v)$.
\item Alice computes $\Exp(A,pk_2)$ and Bob computes $\Exp(B,pk_1)$.
\end{enumerate}
The common key is
\begin{align*}
\Exp(A,pk_2)&=\Exp(A,\Exp(B,v))=\Exp(AB,v)\\
&=\Exp(BA,v)=\Exp(B,\Exp(A,v))=\Exp(B,pk_1).
\end{align*}
\end{protocol}

For computational purposes, Alice and Bob choose $A=\Circ(a_0,\dots,a_{n-1})$ and $B=\Circ(b_0,\dots,b_{n-1})$ with $a_i,b_i\le k$, where $k\in\N$ is a bound that both parties can handle computationally. 

The security of the protocol relies on the following problem.

\begin{problem}\label{prob:CDH}
Let $R$ be a congruence-simple semiring and $M\in\Mat_m(R)$. Given $v$, $\Exp(A,v)$ and $\Exp(B,v)\in\CR[M]^n$ with $A,B\in\Circ_n(\N)$, compute $\Exp(AB,v)$.
\end{problem}

A problem whose solution implies a solution of Problem \ref{prob:CDH} is the following.

\begin{problem}\label{prob:DL}
Let $R$ be a congruence-simple semiring and $M\in\Mat_m(R)$. Given $v$ and $\Exp(A,v)\in\CR[M]^n$ with $A\in\Circ_n(\N)$, find $C\in\Circ_n(\N)$ such that $\Exp(C,v)=\Exp(A,v)$.
\end{problem}

Indeed, if $\Exp(C,v)=\Exp(A,v)$, then
\[
\Exp(C,\Exp(B,v))=\Exp(B,\Exp(C,v))=\Exp(B,\Exp(A,v))
\]
is the shared key.

To prevent a brute-force attack, the set $\CR[M]$ must be large enough. We use the following notions.

\begin{definition}
Let $g$ be an element of a finite monoid and consider the sequence $(g^s)_{s\in\N}$, with $g^0$ the identity. Since the monoid is finite, there exist integers $0\le t<N$ with $g^N=g^t$. Taking $N$ minimal, we call $\ord(g)=N$ the \emph{order}, $\pr(g)=t$ the \emph{preperiod} and $\per(g)=N-t$ the \emph{period} of $g$. The elements $g^0,\dots,g^{N-1}$ are pairwise distinct, and $g^{s+\per(g)}=g^s$ for all $s\ge\pr(g)$.
\end{definition}

To obtain bounds, we use Landau's function
\[
g(n)=\max\{\ord(\sigma):\sigma\in S_n\}=\max\{\lcm(a_1,\dots,a_r): a_i>0,\ a_1+\dots+a_r=n\}.
\]
It is known \cite{Massias,MNR} that
\[
\sqrt{n\ln n}\ \le\ \ln g(n)\ \le\ \sqrt{n\ln n}\,\Bigl(1+\frac{\ln\ln n}{2\ln n}\Bigr),
\]
where the lower bound holds for $n\ge 906$. Finally, the following result is established in \cite{MMR}.

\begin{theorem}[\cite{MMR}]\label{thm:landau}
Let $m\in\N$ and let $R$ be a semiring with $0$ and $1$ and centre $C$. Then there exists a matrix $M\in\Mat_m(R)$ with $\ord(M)\ge g(m)$. In particular, $|C[M]|\ge g(m)$.
\end{theorem}

To obtain suitable matrices, we use the following elementary result.

\begin{lemma}\label{lem:conj}
Let $P,A\in\Mat_n(R)$ with $P$ invertible. Then $\ord(A)=\ord(PAP^{-1})$.
\end{lemma}

\begin{proof}
Since $PAP^{-1}\cdot PAP^{-1}=PA^2P^{-1}$, induction gives $(PAP^{-1})^s=PA^sP^{-1}$ for every $s\in\N$. Hence
\[
(PAP^{-1})^s=(PAP^{-1})^t\iff PA^sP^{-1}=PA^tP^{-1}\iff A^s=A^t,
\]
where the last equivalence follows by multiplying by $P^{-1}$ on the left and by $P$ on the right.
\end{proof}

To proceed, we need to understand the invertible matrices over $R$. To this end, we first introduce sums and direct sums of semigroups.

\begin{definition}
Let $(R,+)$ be a finite semigroup and let $A,B\le R$ be subsemigroups. The \emph{sum} of $A$ and $B$ is $A+B=\{a+b: a\in A,\ b\in B\}$.
\end{definition}

\begin{definition}
Let $(R,+)$ be a semigroup and let $A,B\le R$ be subsemigroups such that $A\cap B=\emptyset$ and every $c\in A+B$ can be written as $c=a+b$ with unique $a\in A$ and $b\in B$. Then $A+B$ is called a \emph{direct sum} and is denoted by $A\oplus B$.
\end{definition}

Direct sums allow us to define reducible and irreducible semigroups.

\begin{definition}
A finite semigroup $(R,+)$ is \emph{reducible} if there exist proper subsemigroups $A,B\subsetneq R$ such that $R=A\oplus B$. Otherwise, $R$ is \emph{irreducible}.
\end{definition}

We now give a sufficient condition for the irreducibility of a finite semigroup. We use the following elementary fact.

\begin{lemma}\label{lem:card}
Let $(R,+)$ be a finite reducible semigroup of cardinality $n$ with $R=A\oplus B$. Then $1<|A|,|B|<n$.
\end{lemma}

With help of the previous lemma we can prove the following new result.

\begin{theorem}\label{thm:irred}
Let $(R,+)$ be a finite semigroup. Suppose that there exists $x\in R$ such that
\begin{enumerate}
\item $z+y=x$ implies $z=x$ or $y=x$, and
\item $x+y=y+x=x$ for all $y\in R$.
\end{enumerate}
Then $R$ is irreducible.
\end{theorem}

\begin{proof}
Suppose, for a contradiction, that there exist subsemigroups $A,B$ with $R=A\oplus B$. Then there is $(a,b)\in A\times B$ with $a+b=x$, and by the first property $a=x$ or $b=x$. Without loss of generality, assume $x\in A$. By Lemma \ref{lem:card}, $|B|\ge2$, so there exist distinct elements $b,d\in B$. By the second property, $x+b=x+d=x$, which contradicts the uniqueness of the decomposition. Hence $R$ is irreducible.
\end{proof}

\begin{example}\label{ex:semiring}
The following semiring with 20 elements was introduced in \cite{MMR} for cryptographic applications.

\begin{center}\small
{\setlength{\tabcolsep}{2.6pt}\begin{tabular}{c|*{20}{c}}$+$ & 0 & a & b & c & d & e & f & g & h & i & j & k & l & m & n & o & p & q & r & 1\\\hline 0 & 0 & a & b & c & d & e & f & g & h & i & j & k & l & m & n & o & p & q & r & 1\\a & a & a & b & c & d & e & f & g & h & i & j & k & l & m & n & o & p & q & r & 1\\b & b & b & b & c & e & e & f & g & h & i & k & k & l & m & n & o & p & q & r & 1\\c & c & c & c & c & f & f & f & h & h & i & l & l & l & 1 & n & p & p & q & r & 1\\d & d & d & e & f & d & e & f & g & h & i & j & k & l & m & n & o & p & q & r & 1\\e & e & e & e & f & e & e & f & g & h & i & k & k & l & m & n & o & p & q & r & 1\\f & f & f & f & f & f & f & f & h & h & i & l & l & l & 1 & n & p & p & q & r & 1\\g & g & g & g & h & g & g & h & g & h & i & m & m & 1 & m & n & o & p & q & r & 1\\h & h & h & h & h & h & h & h & h & h & i & 1 & 1 & 1 & 1 & n & p & p & q & r & 1\\i & i & i & i & i & i & i & i & i & i & i & n & n & n & n & n & q & q & q & r & n\\j & j & j & k & l & j & k & l & m & 1 & n & j & k & l & m & n & o & p & q & r & 1\\k & k & k & k & l & k & k & l & m & 1 & n & k & k & l & m & n & o & p & q & r & 1\\l & l & l & l & l & l & l & l & 1 & 1 & n & l & l & l & 1 & n & p & p & q & r & 1\\m & m & m & m & 1 & m & m & 1 & m & 1 & n & m & m & 1 & m & n & o & p & q & r & 1\\n & n & n & n & n & n & n & n & n & n & n & n & n & n & n & n & q & q & q & r & n\\o & o & o & o & p & o & o & p & o & p & q & o & o & p & o & q & o & p & q & r & p\\p & p & p & p & p & p & p & p & p & p & q & p & p & p & p & q & p & p & q & r & p\\q & q & q & q & q & q & q & q & q & q & q & q & q & q & q & q & q & q & q & r & q\\r & r & r & r & r & r & r & r & r & r & r & r & r & r & r & r & r & r & r & r & r\\1 & 1 & 1 & 1 & 1 & 1 & 1 & 1 & 1 & 1 & n & 1 & 1 & 1 & 1 & n & p & p & q & r & 1\end{tabular}}
\end{center}

\begin{center}\small
{\setlength{\tabcolsep}{2.6pt}\begin{tabular}{c|*{20}{c}}$\cdot$ & 0 & a & b & c & d & e & f & g & h & i & j & k & l & m & n & o & p & q & r & 1\\\hline 0 & 0 & 0 & 0 & 0 & 0 & 0 & 0 & 0 & 0 & 0 & 0 & 0 & 0 & 0 & 0 & 0 & 0 & 0 & 0 & 0\\a & 0 & 0 & 0 & 0 & 0 & 0 & 0 & 0 & 0 & 0 & a & a & a & a & a & b & b & b & c & a\\b & 0 & 0 & 0 & 0 & a & a & a & b & b & c & a & a & a & b & c & b & b & c & c & b\\c & 0 & a & b & c & a & b & c & b & c & c & a & b & c & b & c & b & c & c & c & c\\d & 0 & 0 & 0 & 0 & 0 & 0 & 0 & 0 & 0 & 0 & d & d & d & d & d & g & g & g & i & d\\e & 0 & 0 & 0 & 0 & a & a & a & b & b & c & d & d & d & e & f & g & g & h & i & e\\f & 0 & a & b & c & a & b & c & b & c & c & d & e & f & e & f & g & h & h & i & f\\g & 0 & 0 & 0 & 0 & d & d & d & g & g & i & d & d & d & g & i & g & g & i & i & g\\h & 0 & a & b & c & d & e & f & g & h & i & d & e & f & g & i & g & h & i & i & h\\i & 0 & d & g & i & d & g & i & g & i & i & d & g & i & g & i & g & i & i & i & i\\j & 0 & 0 & 0 & 0 & 0 & 0 & 0 & 0 & 0 & 0 & j & j & j & j & j & o & o & o & r & j\\k & 0 & 0 & 0 & 0 & a & a & a & b & b & c & j & j & j & k & l & o & o & p & r & k\\l & 0 & a & b & c & a & b & c & b & c & c & j & k & l & k & l & o & p & p & r & l\\m & 0 & 0 & 0 & 0 & d & d & d & g & g & i & j & j & j & m & n & o & o & q & r & m\\n & 0 & d & g & i & d & g & i & g & i & i & j & m & n & m & n & o & q & q & r & n\\o & 0 & 0 & 0 & 0 & j & j & j & o & o & r & j & j & j & o & r & o & o & r & r & o\\p & 0 & a & b & c & j & k & l & o & p & r & j & k & l & o & r & o & p & r & r & p\\q & 0 & d & g & i & j & m & n & o & q & r & j & m & n & o & r & o & q & r & r & q\\r & 0 & j & o & r & j & o & r & o & r & r & j & o & r & o & r & o & r & r & r & r\\1 & 0 & a & b & c & d & e & f & g & h & i & j & k & l & m & n & o & p & q & r & 1\end{tabular}}
\end{center}

The additive semigroup $(R,+)$ is irreducible, since $r$ satisfies the conditions of Theorem \ref{thm:irred}.
\end{example}

\begin{definition}
Let $R$ be a semiring and $A\in\Mat_n(R)$. We say that $A$ is a \emph{generalised permutation matrix} if each row and each column contains exactly one nonzero entry and this entry is invertible.
\end{definition}

Kendziorra, Schmidt and Zumbrägel \cite{KSZ} established the following relation between the additive structure of a semiring and the invertibility of matrices over it.

\begin{theorem}[\cite{KSZ}]
Let $(R,+,\cdot)$ be a finite semiring such that $(R,+)$ is irreducible and let $A\in\Mat_n(R)$. Then $A$ is invertible if and only if $A$ is a generalised permutation matrix.
\end{theorem}

We now describe how to obtain the matrices we need. We take a congruence-simple semiring $R$ and a natural number $m$, and we choose positive integers $a_1,\dots,a_r$ with $a_1+\dots+a_r\le m$ such that $\lcm(a_1,\dots,a_r)$ is sufficiently large. Let $T_{a}$ denote the $a\times a$ cyclic permutation matrix and $s=m-\sum_i a_i$. As in Theorem \ref{thm:landau}, we consider the block-diagonal matrix
\[
\begin{pmatrix}
T_{a_1} & 0 & \cdots & 0 & 0\\
0 & T_{a_2} & \cdots & 0 & 0\\
\vdots & \vdots & \ddots & \vdots & \vdots\\
0 & 0 & \cdots & T_{a_r} & 0\\
0 & 0 & \cdots & 0 & I_s
\end{pmatrix}.
\]
Next, we modify the entries above the block diagonal and obtain a matrix of the form
\[
M=\begin{pmatrix}
T_{a_1} & A_{1,2} & A_{1,3} & \cdots & A_{1,r} & A_{1,r+1}\\
0 & T_{a_2} & A_{2,3} & \cdots & A_{2,r} & A_{2,r+1}\\
0 & 0 & T_{a_3} & \cdots & A_{3,r} & A_{3,r+1}\\
\vdots & \vdots & \vdots & \ddots & \vdots & \vdots\\
0 & 0 & 0 & \cdots & T_{a_r} & A_{r,r+1}\\
0 & 0 & 0 & \cdots & 0 & I_s
\end{pmatrix},
\]
where $A_{i,j}\in\Mat_{a_i\times a_j}(R)$ are rectangular matrices with entries in $R$, where we assume $a_{r+1}=s$. Since the diagonal blocks and the zero blocks below them are unchanged, the powers of this matrix have the form
\[
M^t=\begin{pmatrix}
T_{a_1}^t & B_{1,2} & B_{1,3} & \cdots & B_{1,r} & B_{1,r+1}\\
0 & T_{a_2}^t & B_{2,3} & \cdots & B_{2,r} & B_{2,r+1}\\
0 & 0 & T_{a_3}^t & \cdots & B_{3,r} & B_{3,r+1}\\
\vdots & \vdots & \vdots & \ddots & \vdots & \vdots\\
0 & 0 & 0 & \cdots & T_{a_r}^t & B_{r,r+1}\\
0 & 0 & 0 & \cdots & 0 & I_s
\end{pmatrix}
\]
with $B_{i,j}\in\Mat_{a_i\times a_j}(R)$. In particular, $\ord(M)\ge\lcm(a_1,\dots,a_r)$.

By the previous results, the only invertible matrices are the generalised permutation matrices. We can therefore permute the rows and columns of $M$ by conjugation, which preserves the order by Lemma \ref{lem:conj} and makes the block structure, and therefore the computation of powers, harder to recognise. As this changes are in bijection with $S_m$, there are $m!$ possible movements.  For $m$ large enough, a brute force attack to invert this process is not computationally feasible.

This process produces matrices in which roughly half of the entries are zero. This reduces the number of possible public matrices compared with the set of all square matrices. Nevertheless, as shown by the previous results, the resulting set remains large enough for our purposes. Furthermore, this sparsity has a practical advantage: it reduces the amount of information that must be stored and transmitted to describe $M$, since zero entries need not be represented explicitly. For clarity, the examples in this work display the complete matrices; in an implementation, the zero positions can be omitted, which leads to a more compact representation of the public parameters.

\begin{example}\label{ex:6}
We take the 20-element semiring of Example \ref{ex:semiring} and $m=6$, together with the partition $[2,3]$ (and $s=1$). Then $\lcm(2,3)=6$, and we obtain the block matrix
\[
\smallmat{0 & 1 & 0 & 0 & 0 & 0\\1 & 0 & 0 & 0 & 0 & 0\\0 & 0 & 0 & 1 & 0 & 0\\0 & 0 & 0 & 0 & 1 & 0\\0 & 0 & 1 & 0 & 0 & 0\\0 & 0 & 0 & 0 & 0 & 1}.
\]
Modifying the entries above the block diagonal at random, we obtain
\[
\smallmat{0 & 1 & 0 & 0 & 0 & 0\\1 & 0 & 1 & r & b & l\\0 & 0 & 0 & 1 & 0 & e\\0 & 0 & 0 & 0 & 1 & 0\\0 & 0 & 1 & 0 & 0 & 0\\0 & 0 & 0 & 0 & 0 & 1}.
\]
Now, we conjugate this matrix by a generalized permutation matrix. After
performing the calculations, we obtain
\begin{equation*}
    \begin{array}{c@{\qquad\qquad}c} \text{Output matrix} & \text{Generalised permutation matrix}\\[4pt] \smallmat{ 0 & 0 & 0 & 1 & 0 & 0\\ 0 & 0 & 1 & 0 & 0 & e\\ 0 & 0 & 0 & 0 & 1 & 0\\ 1 & 1 & r & 0 & b & l\\ 0 & 1 & 0 & 0 & 0 & 0\\ 0 & 0 & 0 & 0 & 0 & 1 } & \smallmat{ 1 & 0 & 0 & 0 & 0 & 0\\ 0 & 0 & 1 & 0 & 0 & 0\\ 0 & 0 & 0 & 1 & 0 & 0\\ 0 & 1 & 0 & 0 & 0 & 0\\ 0 & 0 & 0 & 0 & 1 & 0\\ 0 & 0 & 0 & 0 & 0 & 1 } \end{array}
\end{equation*}

\end{example}

\subsection{Example}\label{sec:example}

Using the previous results, we now present an small example of the key-exchange protocol. We denote by $M$ the base matrix, by $T=(T_0,T_1)$ the public commutative tuple, by $C_1,C_2$ the private keys, by $pk_1,pk_2$ the public information of each party and by $K$ the common key. We take

\begin{equation}M=
\resizebox{0.5\textwidth}{!}{%
 $\left( \begin{array}{ccccccccccccccccccccc}
0&0&0&0&0&0&0&0&0&0&0&1&0&0&0&0&j&0&0&0&0\\
0&0&0&0&0&0&0&0&0&1&0&0&0&0&0&0&0&0&0&0&0\\
0&0&0&0&0&0&0&0&0&e&0&0&1&0&0&k&0&0&0&0&0\\
0&0&0&0&0&0&1&0&0&0&0&0&0&0&0&0&0&0&0&0&0\\
0&0&k&1&0&0&0&0&0&e&p&0&0&0&b&p&r&0&0&0&0\\
0&0&0&0&1&0&0&0&0&0&0&0&0&0&0&0&0&0&0&0&0\\
0&a&0&0&0&0&0&1&0&0&0&0&0&0&0&0&0&0&0&0&0\\
0&0&0&0&0&0&0&0&1&0&0&0&m&0&0&0&j&0&0&0&0\\
0&0&0&0&0&0&0&0&0&0&0&0&0&1&0&0&0&0&0&0&0\\
0&0&0&0&0&0&0&0&0&0&0&0&0&0&0&1&0&0&0&0&0\\
0&0&0&0&0&0&0&0&0&0&1&0&0&0&0&0&0&0&0&0&0\\
0&0&0&0&0&0&0&0&0&0&0&0&0&0&0&r&0&1&0&0&0\\
1&j&0&0&0&0&0&0&0&0&0&0&0&0&0&0&0&0&0&o&0\\
0&0&0&0&0&0&0&0&0&0&0&0&0&0&0&0&0&0&0&0&1\\
0&1&0&0&0&0&0&0&0&0&0&0&0&0&0&0&0&0&0&0&0\\
0&0&0&0&0&0&0&0&0&0&0&0&0&0&0&0&1&0&0&0&0\\
0&0&0&0&0&0&0&0&0&0&0&0&0&0&0&0&0&0&1&0&0\\
0&0&1&0&0&0&0&0&0&0&0&0&0&0&0&0&r&0&g&0&0\\
0&0&0&0&0&0&0&0&0&0&0&0&0&0&0&0&0&0&0&1&0\\
0&0&0&0&0&0&0&0&0&0&0&0&0&0&1&0&0&0&0&0&0\\
0&0&0&0&0&1&0&0&0&0&0&a&r&0&0&0&g&0&0&0&0 
\end{array} \right)$
}
\end{equation}

\[
C_1=\begin{pmatrix}4&1\\
1&4 \end{pmatrix},\qquad
C_2=\begin{pmatrix}5&10\\
10&5 \end{pmatrix},
\]
and the tuple $T$ given by
\par\medskip\noindent\begin{minipage}{0.49\textwidth}\centering $T_{0}$\\[2pt]\scriptsize $\smallmat{0&0&l&i&0&0&0&0&0&0&0&0&i&0&i&0&0&0&0&0&1\\
0&0&0&0&1&0&0&0&0&0&0&0&0&0&0&0&0&0&0&0&0\\
0&1&0&0&0&0&0&0&0&0&0&0&0&0&0&0&0&0&0&0&0\\
0&0&0&0&0&0&0&0&0&0&0&0&0&0&1&0&0&0&0&0&0\\
0&0&0&0&0&0&0&0&1&0&0&0&0&0&0&0&0&0&0&0&0\\
0&m&i&n&p&0&0&n&0&0&0&0&0&0&0&0&0&1&0&0&0\\
0&q&0&0&0&0&0&0&0&0&1&c&0&0&l&0&p&0&0&f&0\\
0&0&r&0&p&0&0&0&j&0&0&f&0&1&g&0&0&0&0&0&0\\
0&0&0&1&0&0&0&0&0&0&0&0&0&0&0&0&0&0&0&0&0\\
0&0&1&0&0&1&0&c&0&0&0&0&0&a&0&0&q&0&f&0&0\\
1&0&1&0&h&0&0&n&e&0&0&o&0&0&0&0&1&0&o&k&0\\
0&0&0&0&0&0&0&0&0&0&0&1&0&0&0&0&0&0&0&0&0\\
0&0&0&0&0&0&0&0&r&0&0&0&0&0&p&0&a&0&1&0&0\\
0&0&i&l&l&0&0&0&0&0&0&0&1&0&i&0&0&0&0&0&0\\
0&0&0&0&0&0&0&0&0&0&0&0&0&0&0&0&1&0&0&0&0\\
0&0&0&1&0&0&0&0&r&1&0&o&k&o&p&0&0&0&j&i&0\\
0&0&1&0&0&0&0&0&0&0&0&0&0&0&0&0&0&0&0&0&0\\
0&0&0&c&1&0&0&0&0&0&0&c&i&0&p&1&e&0&0&0&0\\
0&0&l&0&g&0&0&0&i&0&0&f&0&0&l&0&0&0&0&1&0\\
0&c&0&0&i&0&0&1&0&0&0&0&0&0&0&0&0&0&0&0&0\\
0&0&a&1&0&0&1&0&q&0&0&c&n&f&l&0&0&0&c&a&0 }$\end{minipage}\hfill\begin{minipage}{0.49\textwidth}\centering $T_{1}$\\[2pt]\scriptsize $\smallmat{0&r&q&q&q&0&0&n&r&1&0&q&n&q&q&0&q&1&n&i&0\\
0&0&1&0&1&0&0&0&0&0&0&0&0&0&0&0&0&0&0&0&0\\
0&1&0&0&0&0&0&0&0&0&0&0&0&0&0&0&1&0&0&0&0\\
0&0&0&0&0&0&0&0&1&0&0&0&0&0&1&0&0&0&0&0&0\\
0&1&0&0&0&0&0&0&1&0&0&0&0&0&0&0&0&0&0&0&0\\
1&r&r&r&r&0&1&q&q&0&0&q&n&n&n&0&r&0&q&n&0\\
0&r&q&r&q&0&0&n&r&1&0&q&k&q&r&0&r&1&q&n&0\\
0&r&r&n&q&0&0&1&p&0&0&f&1&0&r&0&n&0&0&0&0\\
0&0&0&1&1&0&0&0&0&0&0&0&0&0&0&0&0&0&0&0&0\\
0&q&r&q&q&0&0&n&r&0&1&q&i&n&r&0&q&0&n&p&1\\
0&q&r&r&r&1&0&n&r&0&0&q&q&n&r&1&r&0&l&q&0\\
0&0&0&0&0&0&0&0&0&0&0&1&0&0&0&0&0&0&0&0&0\\
0&r&n&r&n&0&0&0&r&0&0&f&1&0&q&0&p&0&0&1&0\\
0&i&r&p&r&0&0&0&r&0&0&f&0&1&q&0&i&0&1&0&0\\
0&0&0&1&0&0&0&0&0&0&0&0&0&0&0&0&1&0&0&0&0\\
1&r&r&r&r&0&1&q&r&0&0&q&q&l&r&0&r&0&q&n&0\\
0&0&1&0&0&0&0&0&0&0&0&0&0&0&1&0&0&0&0&0&0\\
0&r&r&q&r&0&0&k&r&0&1&q&n&q&r&0&q&0&n&q&1\\
0&n&p&i&i&0&0&1&r&0&0&f&0&0&r&0&n&0&1&0&0\\
0&p&r&r&q&0&0&0&n&0&0&f&0&1&n&0&r&0&0&1&0\\
0&q&r&r&r&1&0&i&r&0&0&q&q&n&r&1&q&0&n&n&0}$\end{minipage}\par\medskip
\par\medskip\noindent
The public keys are $pk_1=(pk_{1,1},pk_{1,2})$ and $pk_2=(pk_{2,1},pk_{2,2})$, with
\par\medskip\noindent\begin{minipage}{0.49\textwidth}\centering $pk_{1,1}$\\[2pt]\scriptsize $\smallmat{0&r&r&r&r&0&0&q&r&1&0&q&n&q&r&0&r&1&q&n&0\\
0&0&0&1&0&0&0&0&0&0&0&0&0&0&0&0&1&0&0&0&0\\
0&0&0&0&0&0&0&0&1&0&0&0&0&0&1&0&0&0&0&0&0\\
0&0&1&0&1&0&0&0&0&0&0&0&0&0&0&0&0&0&0&0&0\\
0&0&1&0&0&0&0&0&0&0&0&0&0&0&1&0&0&0&0&0&0\\
1&r&r&r&r&0&1&q&r&0&0&q&q&n&r&0&r&0&q&q&0\\
0&r&r&r&r&0&0&q&r&1&0&q&q&q&r&0&r&1&q&n&0\\
0&q&r&r&r&0&0&0&r&0&0&f&0&1&r&0&r&0&0&1&0\\
0&1&0&0&0&0&0&0&0&0&0&0&0&0&0&0&1&0&0&0&0\\
0&r&r&r&r&0&0&n&r&0&1&q&n&q&r&0&r&0&q&q&1\\
0&r&r&r&r&1&0&n&r&0&0&q&q&q&r&1&r&0&q&q&0\\
0&0&0&0&0&0&0&0&0&0&0&1&0&0&0&0&0&0&0&0&0\\
0&r&r&q&r&0&0&0&r&0&0&f&0&1&r&0&r&0&1&0&0\\
0&r&r&r&q&0&0&1&q&0&0&f&1&0&r&0&r&0&0&0&0\\
0&1&0&0&0&0&0&0&1&0&0&0&0&0&0&0&0&0&0&0&0\\
1&r&r&r&r&0&1&q&r&0&0&q&q&q&r&0&r&0&q&q&0\\
0&0&0&1&1&0&0&0&0&0&0&0&0&0&0&0&0&0&0&0&0\\
0&r&r&r&r&0&0&q&r&0&1&q&n&q&r&0&r&0&q&q&1\\
0&r&r&r&r&0&0&0&r&0&0&f&1&0&q&0&q&0&0&1&0\\
0&r&q&r&r&0&0&1&r&0&0&f&0&0&r&0&q&0&1&0&0\\
0&r&r&r&r&1&0&n&r&0&0&q&q&q&r&1&r&0&n&q&0  }$\end{minipage}\hfill\begin{minipage}{0.49\textwidth}\centering $pk_{1,2}$\\[2pt]\scriptsize $\smallmat{0&r&r&r&r&0&0&q&r&0&1&q&q&q&r&0&r&0&q&q&1\\
0&0&1&1&1&0&0&0&0&0&0&0&0&0&1&0&1&0&0&0&0\\
0&1&0&1&0&0&0&0&1&0&0&0&0&0&1&0&1&0&0&0&0\\
0&1&1&0&1&0&0&0&1&0&0&0&0&0&1&0&0&0&0&0&0\\
0&1&1&0&0&0&0&0&1&0&0&0&0&0&1&0&1&0&0&0&0\\
0&r&r&r&r&0&0&q&r&1&0&q&q&q&r&0&r&1&q&q&0\\
0&r&r&r&r&0&0&q&r&0&1&q&q&q&r&0&r&0&q&q&1\\
0&r&r&r&r&0&0&1&r&0&0&f&1&1&r&0&r&0&1&1&0\\
0&1&1&1&1&0&0&0&0&0&0&0&0&0&0&0&1&0&0&0&0\\
0&r&r&r&r&1&0&q&r&0&0&q&q&q&r&1&r&0&q&q&0\\
1&r&r&r&r&0&1&q&r&0&0&q&q&q&r&0&r&0&q&q&0\\
0&0&0&0&0&0&0&0&0&0&0&1&0&0&0&0&0&0&0&0&0\\
0&r&r&r&r&0&0&1&r&0&0&f&1&1&r&0&r&0&1&1&0\\
0&r&r&r&r&0&0&1&r&0&0&f&1&1&r&0&r&0&1&1&0\\
0&1&0&1&1&0&0&0&1&0&0&0&0&0&0&0&1&0&0&0&0\\
0&r&r&r&r&0&0&q&r&1&0&q&q&q&r&0&r&1&q&q&0\\
0&0&1&1&1&0&0&0&1&0&0&0&0&0&1&0&0&0&0&0&0\\
0&r&r&r&r&1&0&q&r&0&0&q&q&q&r&1&r&0&q&q&0\\
0&r&r&r&r&0&0&1&r&0&0&f&1&1&r&0&r&0&1&1&0\\
0&r&r&r&r&0&0&1&r&0&0&f&1&1&r&0&r&0&1&1&0\\
1&r&r&r&r&0&1&q&r&0&0&q&q&q&r&0&r&0&q&q&0 }$\end{minipage}\par\medskip
\par\medskip\noindent
\par\medskip\noindent\begin{minipage}{0.49\textwidth}\centering $pk_{2,1}$\\[2pt]\scriptsize $\smallmat{1&r&r&r&r&0&1&q&r&0&0&q&q&q&r&0&r&0&q&q&0\\
0&1&1&1&1&0&0&0&1&0&0&0&0&0&1&0&1&0&0&0&0\\
0&1&1&1&1&0&0&0&1&0&0&0&0&0&1&0&1&0&0&0&0\\
0&1&1&1&1&0&0&0&1&0&0&0&0&0&1&0&1&0&0&0&0\\
0&1&1&1&1&0&0&0&1&0&0&0&0&0&1&0&1&0&0&0&0\\
0&r&r&r&r&1&0&q&r&0&0&q&q&q&r&1&r&0&q&q&0\\
1&r&r&r&r&0&1&q&r&0&0&q&q&q&r&0&r&0&q&q&0\\
0&r&r&r&r&0&0&1&r&0&0&f&1&1&r&0&r&0&1&1&0\\
0&1&1&1&1&0&0&0&1&0&0&0&0&0&1&0&1&0&0&0&0\\
0&r&r&r&r&0&0&q&r&1&0&q&q&q&r&0&r&1&q&q&0\\
0&r&r&r&r&0&0&q&r&0&1&q&q&q&r&0&r&0&q&q&1\\
0&0&0&0&0&0&0&0&0&0&0&1&0&0&0&0&0&0&0&0&0\\
0&r&r&r&r&0&0&1&r&0&0&f&1&1&r&0&r&0&1&1&0\\
0&r&r&r&r&0&0&1&r&0&0&f&1&1&r&0&r&0&1&1&0\\
0&1&1&1&1&0&0&0&1&0&0&0&0&0&1&0&1&0&0&0&0\\
0&r&r&r&r&1&0&q&r&0&0&q&q&q&r&1&r&0&q&q&0\\
0&1&1&1&1&0&0&0&1&0&0&0&0&0&1&0&1&0&0&0&0\\
0&r&r&r&r&0&0&q&r&1&0&q&q&q&r&0&r&1&q&q&0\\
0&r&r&r&r&0&0&1&r&0&0&f&1&1&r&0&r&0&1&1&0\\
0&r&r&r&r&0&0&1&r&0&0&f&1&1&r&0&r&0&1&1&0\\
0&r&r&r&r&0&0&q&r&0&1&q&q&q&r&0&r&0&q&q&1   }$\end{minipage}\hfill\begin{minipage}{0.49\textwidth}\centering $pk_{2,2}$\\[2pt]\scriptsize $\smallmat{0&r&r&r&r&0&0&q&r&1&0&q&q&q&r&0&r&1&q&q&0\\
0&1&1&0&1&0&0&0&1&0&0&0&0&0&1&0&1&0&0&0&0\\
0&1&1&1&1&0&0&0&0&0&0&0&0&0&1&0&1&0&0&0&0\\
0&1&0&1&1&0&0&0&1&0&0&0&0&0&1&0&1&0&0&0&0\\
0&1&1&1&1&0&0&0&1&0&0&0&0&0&0&0&1&0&0&0&0\\
1&r&r&r&r&0&1&q&r&0&0&q&q&q&r&0&r&0&q&q&0\\
0&r&r&r&r&0&0&q&r&1&0&q&q&q&r&0&r&1&q&q&0\\
0&r&r&r&r&0&0&1&r&0&0&f&1&1&r&0&r&0&1&1&0\\
0&1&1&1&1&0&0&0&1&0&0&0&0&0&1&0&0&0&0&0&0\\
0&r&r&r&r&0&0&q&r&0&1&q&q&q&r&0&r&0&q&q&1\\
0&r&r&r&r&1&0&q&r&0&0&q&q&q&r&1&r&0&q&q&0\\
0&0&0&0&0&0&0&0&0&0&0&1&0&0&0&0&0&0&0&0&0\\
0&r&r&r&r&0&0&1&r&0&0&f&1&1&r&0&r&0&1&1&0\\
0&r&r&r&r&0&0&1&r&0&0&f&1&1&r&0&r&0&1&1&0\\
0&0&1&1&1&0&0&0&1&0&0&0&0&0&1&0&1&0&0&0&0\\
1&r&r&r&r&0&1&q&r&0&0&q&q&q&r&0&r&0&q&q&0\\
0&1&1&1&0&0&0&0&1&0&0&0&0&0&1&0&1&0&0&0&0\\
0&r&r&r&r&0&0&q&r&0&1&q&q&q&r&0&r&0&q&q&1\\
0&r&r&r&r&0&0&1&r&0&0&f&1&1&r&0&r&0&1&1&0\\
0&r&r&r&r&0&0&1&r&0&0&f&1&1&r&0&r&0&1&1&0\\
0&r&r&r&r&1&0&q&r&0&0&q&q&q&r&1&r&0&q&q&0 }$\end{minipage}\par\medskip
\par\medskip\noindent
and the common key $K=(K_1,K_2)$ is
\par\medskip\noindent\begin{minipage}{0.49\textwidth}\centering $K_1$\\[2pt]\scriptsize $\smallmat{0&r&r&r&r&0&0&q&r&1&0&q&q&q&r&0&r&1&q&q&0\\
0&1&1&1&1&0&0&0&1&0&0&0&0&0&1&0&1&0&0&0&0\\
0&1&1&1&1&0&0&0&1&0&0&0&0&0&1&0&1&0&0&0&0\\
0&1&1&1&1&0&0&0&1&0&0&0&0&0&1&0&1&0&0&0&0\\
0&1&1&1&1&0&0&0&1&0&0&0&0&0&1&0&1&0&0&0&0\\
1&r&r&r&r&0&1&q&r&0&0&q&q&q&r&0&r&0&q&q&0\\
0&r&r&r&r&0&0&q&r&1&0&q&q&q&r&0&r&1&q&q&0\\
0&r&r&r&r&0&0&1&r&0&0&f&1&1&r&0&r&0&1&1&0\\
0&1&1&1&1&0&0&0&1&0&0&0&0&0&1&0&1&0&0&0&0\\
0&r&r&r&r&0&0&q&r&0&1&q&q&q&r&0&r&0&q&q&1\\
0&r&r&r&r&1&0&q&r&0&0&q&q&q&r&1&r&0&q&q&0\\
0&0&0&0&0&0&0&0&0&0&0&1&0&0&0&0&0&0&0&0&0\\
0&r&r&r&r&0&0&1&r&0&0&f&1&1&r&0&r&0&1&1&0\\
0&r&r&r&r&0&0&1&r&0&0&f&1&1&r&0&r&0&1&1&0\\
0&1&1&1&1&0&0&0&1&0&0&0&0&0&1&0&1&0&0&0&0\\
1&r&r&r&r&0&1&q&r&0&0&q&q&q&r&0&r&0&q&q&0\\
0&1&1&1&1&0&0&0&1&0&0&0&0&0&1&0&1&0&0&0&0\\
0&r&r&r&r&0&0&q&r&0&1&q&q&q&r&0&r&0&q&q&1\\
0&r&r&r&r&0&0&1&r&0&0&f&1&1&r&0&r&0&1&1&0\\
0&r&r&r&r&0&0&1&r&0&0&f&1&1&r&0&r&0&1&1&0\\
0&r&r&r&r&1&0&q&r&0&0&q&q&q&r&1&r&0&q&q&0    }$\end{minipage}\hfill\begin{minipage}{0.49\textwidth}\centering $K_2$\\[2pt]\scriptsize $\smallmat{1&r&r&r&r&0&1&q&r&0&0&q&q&q&r&0&r&0&q&q&0\\
0&1&1&1&1&0&0&0&1&0&0&0&0&0&1&0&1&0&0&0&0\\
0&1&1&1&1&0&0&0&1&0&0&0&0&0&1&0&1&0&0&0&0\\
0&1&1&1&1&0&0&0&1&0&0&0&0&0&1&0&1&0&0&0&0\\
0&1&1&1&1&0&0&0&1&0&0&0&0&0&1&0&1&0&0&0&0\\
0&r&r&r&r&1&0&q&r&0&0&q&q&q&r&1&r&0&q&q&0\\
1&r&r&r&r&0&1&q&r&0&0&q&q&q&r&0&r&0&q&q&0\\
0&r&r&r&r&0&0&1&r&0&0&f&1&1&r&0&r&0&1&1&0\\
0&1&1&1&1&0&0&0&1&0&0&0&0&0&1&0&1&0&0&0&0\\
0&r&r&r&r&0&0&q&r&1&0&q&q&q&r&0&r&1&q&q&0\\
0&r&r&r&r&0&0&q&r&0&1&q&q&q&r&0&r&0&q&q&1\\
0&0&0&0&0&0&0&0&0&0&0&1&0&0&0&0&0&0&0&0&0\\
0&r&r&r&r&0&0&1&r&0&0&f&1&1&r&0&r&0&1&1&0\\
0&r&r&r&r&0&0&1&r&0&0&f&1&1&r&0&r&0&1&1&0\\
0&1&1&1&1&0&0&0&1&0&0&0&0&0&1&0&1&0&0&0&0\\
0&r&r&r&r&1&0&q&r&0&0&q&q&q&r&1&r&0&q&q&0\\
0&1&1&1&1&0&0&0&1&0&0&0&0&0&1&0&1&0&0&0&0\\
0&r&r&r&r&0&0&q&r&1&0&q&q&q&r&0&r&1&q&q&0\\
0&r&r&r&r&0&0&1&r&0&0&f&1&1&r&0&r&0&1&1&0\\
0&r&r&r&r&0&0&1&r&0&0&f&1&1&r&0&r&0&1&1&0\\
0&r&r&r&r&0&0&q&r&0&1&q&q&q&r&0&r&0&q&q&1  }$\end{minipage}\par\medskip
\par\medskip\noindent

\section{Key-exchange protocol and its security}\label{sec:security}

We begin by analysing the computational cost of the key exchange. We measure all costs in terms of operations in $R$. When $R$ is a finite semiring specified by its operation tables, these operations can be performed directly using the corresponding tables. In the case of the semiring in Example \ref{ex}, each element can be represented using $5$ bits, and each operation table contains $400$ entries. As in the ring case, the product of two $m\times m$ matrices $O(m^3)$ operations.

\subsection{Setup cost}

The setup consists of constructing $M$ and the public tuple $v$. Constructing $M$ as described in Section \ref{sec:math} costs $O(m^2)$ operations and is negligible. To compute the tumple $v=(p_j(M))_{j=0}^{n-1}$, we have $O(hm^3+nhm^2)$ operations.

\begin{theorem}\label{thm:setup}
Let $M\in\Mat_m(R)$ and let $v_j=p_j(M)$, $j=0,\dots,n-1$, where $p_j\in\CR[x]$ has degree at most $h$. Then $v$ can be computed with at most $h-1$ matrix multiplications and $O(nhm^2)$ further operations in $R$, that is, with $O(hm^3+nhm^2)$ operations in $R$.
\end{theorem}

\begin{proof}
The powers $M^2,\dots,M^h$ are computed once, with $h-1$ matrix multiplications, that is, $O(hm^3)$ operations, and are shared by all the $v_j$. Each $v_j=\sum_{t=0}^{h}c_{j,t}M^t$ then requires at most $h+1$ products of a central element by an $m\times m$ matrix and $h$ matrix additions, that is, $O(hm^2)$ operations. Summing over the $n$ components gives the bound.
\end{proof}

\begin{corollary}
    Let $R$ be a semiring with $C(R)={0,1}$. Let $M\in\Mat_m(R)$ and let $v_j=p_j(M)$, $j=0,\dots,n-1$, where $p_j\in\CR[x]$ are sparse polynomial with at most $w$ nonzero terms and degree at most $h$. Then $v$ can be computed with $O(nw\log_2(h)m^3)$ operations in $R$.
\end{corollary}
\begin{proof}
If the polynomials $p_j$ are sparse, with at most $w$ nonzero terms, the powers can instead be computed by repeated squaring, which gives $O(nw\log_2(h)m^3)$ operations. In addition, as $\CR=\{0,1\}$, the scalar products are trivial and each $v_j$ is simply a sum of selected powers of $M$, which are $O(nwm^2)$ operations. Finally, $O(nw\log_2(h)m^3 + nwm^2) = O(nw\log_2(h)m^3)$.
\end{proof}
Note that the condition of the previous corollary are satisfied by the semiring of Example \ref{ex:semiring}, which reduce the computational cost needed in order to compute the tuple. In addition, in any case  the setup is performed only once, and the public parameters $(R,M,v)$ can be fixed and reused by all users, as with standardised groups in classical Diffie--Hellman.

\subsection{Cost of the key exchange}

\begin{theorem}\label{thm:postsetup}
Let $w=(w_0,\dots,w_{n-1})\in\CR[M]^n$, let $A\in\Circ_n([0,k])$ be a circulant matrix with entries in $\{0,\dots,k\}$, and let $L=\lfloor\log_2 k\rfloor$. Then $\Exp(A,w)$ can be computed with at most $n(n+1)L+n(n-1)$ multiplications of $m\times m$ matrices, that is, with $O(n^2\log_2(k)\,m^3)$ operations in $R$.
\end{theorem}

\begin{proof}
For each $j$ we compute $w_j^{2},w_j^{4},\dots,w_j^{2^L}$ by repeated squaring, with $nL$ multiplications in total. Every power $w_j^a$ with $0\le a\le k$ is then the product of at most $L+1$ of these matrices, which requires at most $L$ multiplications. The $i$th component $\prod_{j}w_j^{a_{i-j}}$ therefore requires at most $nL$ multiplications to form the $n$ powers and $n-1$ multiplications to combine them. Summing over the $n$ components gives at most $nL+n(nL+n-1)=n(n+1)L+n(n-1)$ matrix multiplications.
\end{proof}

Each party applies the action twice, once to compute its public key and once to compute the shared key, so the total cost per party is at most $2\bigl(n(n+1)L+n(n-1)\bigr)$ matrix multiplications. The public key consists of $n$ matrices of size $m\times m$, that is, $nm^2\lceil\log_2|R|\rceil$ bits. Note that the zero entrances of $M$ that comes from the elements below the diagonal will remain zero in $C_R[M]$, and therefore saving from sparsity should be expected for the public keys.

\subsection{Comparison and concrete parameters}\label{sec:params}

\subsubsection{Requirements on the parameters.}
The parameters must satisfy the following conditions.
\begin{enumerate}
\item \emph{Effective exponent range.} Since each $v_j$ lies in a finite monoid, $v_j^{a}$ depends only on $a$ modulo $\per(v_j)$ once $a\ge\pr(v_j)$. Choosing $k$ larger than the orders of the $v_j$ therefore does not enlarge the key space: the number of distinct public keys is at most $\min\bigl(k+1,\max_j\ord(v_j)\bigr)^n$. We therefore require $\ord(v_j)\ge k$ for all $j$.
\item \emph{Exhaustive search.} The nominal key space has $(k+1)^n$ elements. For a target security level of $\lambda$ bits against quantum adversaries, taking Grover's quadratic speed-up into account, we require $n\log_2 k\ge2\lambda$; for classical security only, $n\log_2k\ge\lambda$ suffices.
\item \emph{Order of $M$.} Condition~1 requires in particular $\ord(M)$ to be large. By Theorem \ref{thm:landau} and the construction above, a matrix of size $m$ with $\ord(M)\ge g(m)$ is available, so $m$ must satisfy $\log_2 g(m)\ge\log_2k$.
\item \emph{Choice of the tuple.} The components $v_j$ should not all be powers of a single matrix , and condition~1 must be verified for the actual tuple $v$, not only for $M$. The order of an element of $\CR[M]$ can be much smaller than $\ord(M)$, particularly over additively idempotent semirings, where sums of powers tend to stabilise.
\end{enumerate}

Table \ref{tab:params} lists three parameter sets targeting $2\lambda=256$ bits of nominal key space, that is, $\lambda=128$ bits against Grover-type exhaustive search. The values of $m$ are the smallest ones for which $\log_2 g(m)\ge\log_2 k$, computed exactly from Landau's function. Costs assume the bound of Theorem \ref{thm:postsetup} and a semiring with $20$ elements, as in Example \ref{ex:semiring}. These parameters are provisional: they are derived only from exhaustive search and from condition~1, they presuppose that a tuple satisfying condition~4 has been found, and they must be revised if better attacks on Problem \ref{prob:DL} become known.

\begin{table}[ht]
\centering\small
\begin{tabular}{@{}ccccccc@{}}
\toprule
Set & $n$ & $\log_2k$ & $m$ & $\log_2g(m)$ & mult.\ per party & public key\\
\midrule
I   & 16 & 16 & 48  & 16.9 & $1.0\cdot10^{9}$ & 13.5 KiB\\
II  & 8  & 32 & 120 & 32.3 & $8.2\cdot10^{9}$ & 42.18 KiB\\
III & 4  & 64 & 360 & 64.3 & $1.2\cdot10^{11}$ & 189.6 KiB\\
\bottomrule
\end{tabular}
\caption{Parameter sets with $n\log_2k=256$. The column ``mult.\ per party'' gives the number of multiplications in $R$ for the two applications of the action (an equal number of additions is needed), and the public key size is $5$ bits for the non $0$ values and $1$ bit for the $0$ values. Since approximately half of the entries of the matrix are zero, we lead to $nm^2\cdot \left(\frac{5}{2} + \frac{1}{2}\right) = nm^23$.}
\label{tab:params}
\end{table}

Since the cost grows as $n^2\log_2(k)\,m^3$ and $m$ grows roughly quadratically in $\log_2 k$ (because $\ln g(m)\approx\sqrt{m\ln m}$), it is preferable to use many components with moderately sized exponents. Set~I is therefore our recommended choice. For classical security of $128$ bits only, the choice $n=8$, $\log_2 k=16$, $m=48$ halves the key space exponent and roughly divides the cost by four.

\subsubsection{Comparison with the tropical multiple exponentiation scheme}\label{sec:HL}

For $f(x)=x^n-1$, the action of $\mathbb{Z}_{\ge0}[C]$ used by Huang and Li \cite{HL22} coincides with the action $\Exp$ of $\Circ_n(\N)$ considered here. The difference lies in the platform: $m\times m$ matrices over a finite congruence-simple semiring $R$ in our case, and $m\times m$ matrices over the tropical semiring $(\mathbb{Z}\cup\{\infty\},\min,+)$ in \cite{HL22}, with public tuple given by tropical polynomials of degree at most $9$, with coefficients in $[0,d)$, in a hidden matrix $N$ with entries in $[0,\tau]$. As in Theorem \ref{thm:postsetup}, the private circulant matrices have entries in $[0,k]$ in both schemes.

Since the action is the same, both schemes perform the same sequence of matrix operations, and the cost of one application of the action in \cite{HL22} is the same as that given by Theorem \ref{thm:postsetup}, namely $O(n^2\log_2(k)\,m^3)$ elementary operations. 

The private keys have $n\lceil\log_2(k+1)\rceil$ bits in both schemes. Each entry of a matrix over $R$ needs $\lceil\log_2|R|\rceil$ bits, which is $5$ for the semiring of Example~24, independently of $n$ and $k$. Assuming that the polynomials defining the public tuple are sparse about $nm^23$ bits. In \cite{HL22}, the entries grow with the exponents: they are bounded by $nk(d+9\tau)$, so each of the $nm^2$ entries needs $b_{\mathrm{trop}}=\lceil\log_2(nk(d+9\tau)+1)\rceil$ bits, which is $18$ or $19$ for the parameters of \cite{HL22} ($\tau=100$, $d=1024$). Table \ref{tab:HL} compares both schemes for the parameter sets of \cite[Table~2]{HL22}, with the same dimension $m$. We use a direct count of the $nm^2$ entries rather than the upper bound given in \cite{HL22}, which contains a factor $n^2m^2$.

\begin{table}[ht]\label{tab:HL}
\centering\small
\begin{tabular}{@{}cccrr@{}}
\toprule
$m$ & $n$ & $k$ & pk ours (KiB) & pk tropical (KiB)\\
\midrule
9  & 80 & 1 & 2.4 & 14.2\\
14 & 50 & 2 & 3.6 & 21.5\\
19 & 40 & 3 & 5.3 & 31.7\\
23 & 34 & 4 & 6.6 & 39.5\\
28 & 31 & 5 & 8.9 & 56.4\\
\bottomrule
\end{tabular}
\caption{Public key sizes for the parameter sets of \cite[Table~2]{HL22}, which satisfy $(k+1)^n\approx2^{80}$: $5nm^2/3$ bits in our case and $nm^2\,b_{\mathrm{trop}}$ bits in \cite{HL22}.}
\label{tab:HL}
\end{table}

In \cite{HL22}, tropical matrix powers exhibit an eventual linear periodicity, and the authors recommend $m\ge\sqrt n\,k$ to resist heuristic attacks. The dimension therefore grows linearly with the exponent range, which forces small values of $k$ and large values of $n$. In our case, the condition is $\ord(v_j)\ge k$, which by Theorem~17 requires $\log_2 g(m)\ge\log_2 k$. For instance, for $n=16$ and $k=2^{16}$ , we need $m=48$ and public keys of about $13.5$ KiB, whereas the condition of \cite{HL22} requires $m\geq262144$. Since the cost is cubic in $m$, the finite platform allows large exponent ranges at a moderate cost. 

\subsubsection{Comparison with the protocol of \cite{MMR}.}
In the protocol of \cite{MMR}, the public data are matrices $M_1,M_2,X\in\Mat_m(R)$, Alice's private key is a pair of polynomials $p_A,q_A\in\CR[x]$ of degree at most $h$, and her public key is $p_A(M_1)Xq_A(M_2)$. Since $M_1$ and $M_2$ are public, the powers $M_1^i$ and $M_2^i$, $2\le i\le h$, can be computed once during the setup, with $2(h-1)$ matrix multiplications, that is, $O(hm^3)$ operations for the set up. Evaluating $p_A(M_1)=\sum_{i=0}^{h}c_iM_1^i$ and $q_A(M_2)$ then requires only products of central scalars by matrices and matrix additions, that is, $O(hm^2)$ operations in $R$. The public key and the shared key $p_A(M_1)pk_Bq_A(M_2)$ each require two matrix multiplications. Hence, after the setup, each party performs $4$ matrix multiplications and $O(hm^2)$ further operations, that is, $O(m^3+hm^2)$ operations in $R$, and transmits a single $m\times m$ matrix. For the same matrix size $m$, our protocol is therefore more expensive by a factor of roughly $n^2\log_2(k)/2$ in computation and by a factor of $n$ in communication. However, a comparison at equal security level is not possible, since the attack of \cite{OL} breaks the protocol of \cite{MMR} in time polynomial in $h$ and $m$ for every finite congruence-simple semiring; no choice of parameters makes it secure. The additional cost of our construction is the price of an action that does not have the linear structure exploited in \cite{OL}.

\subsubsection{Comparison with classical and standardised schemes.}
Classical elliptic-curve Diffie--Hellman, for instance X25519 \cite{RFC7748}, uses $32$-byte public keys and requires well under a millisecond on commodity hardware, and the standardised post-quantum key-encapsulation mechanism ML-KEM-768 \cite{FIPS203} uses $1184$-byte encapsulation keys and $1088$-byte ciphertexts. As Table \ref{tab:params} shows, the protocol proposed here requires public keys of tens of kilobytes and on the order of $10^9$ semiring operations per party for a nominal security level of $128$ bits. It therefore does not compare positively with these schemes in terms of efficiency. The interest of the construction lies in the underlying hardness assumption, which is of a different nature from those of the schemes mentioned above, rather than in its performance.

\subsection{Brute-force and random attacks}\label{sec:brute}

An attacker may try a brute-force or random attack against the protocol. If $k$ is an upper bound for the entries of the circulant matrices, a brute-force attack must try up to $(k+1)^n$ different matrices. To prevent this, $k$ and $n$ must be chosen so that $(k+1)^n$ is infeasible large.

However, this is not sufficient: if the set $\{C\in\Circ_n(\N): \Exp(C,v)=\Exp(A,v)\}$ of equivalent private keys is too large, a brute-force or random search could easily find one of its elements. Computational experiments show that a random attack is not effective against matrices over the semiring of example \ref{ex:semiring}. In addition, the following results show that, under certain conditions, the private key is unique.

\begin{theorem}\label{thm:unique1}
Let $n=1$, $v=(M)$ and $A=(a)$ with $a\le\ord(M)-1$. Then $X=A$ is the only solution of $\Exp(X,v)=\Exp(A,v)$ with $X=(x)$, $x\le\ord(M)-1$.
\end{theorem}

\begin{proof}
The equation reads $M^x=M^a$, and the powers $M^0,\dots,M^{\ord(M)-1}$ are pairwise distinct.
\end{proof}

\begin{theorem}\label{thm:unique2}
Let $n\in\N$, let $v=(M^{b_0},\dots,M^{b_{n-1}})$ with $b=(b_0,\dots,b_{n-1})^{T}\in\N^n$ such that $\det\Circ(b_0,\dots,b_{n-1})\ne0$, and let $A=\Circ(a_0,\dots,a_{n-1})\in\Circ_n(\N)$ be such that $(Ab)_i\le\ord(M)-1$ for all $i$. Then $X=A$ is the only solution of $\Exp(X,v)=\Exp(A,v)$ among the matrices $X\in\Circ_n(\N)$ with $(Xb)_i\le\ord(M)-1$ for all $i$.
\end{theorem}

\begin{proof}
We have $\Exp(A,v)_i=\prod_j M^{a_{i-j}b_j}=M^{(Ab)_i}$. Let $X=\Circ(y_0,\dots,y_{n-1})$ be another solution satisfying the bound. Then $M^{(Xb)_i}=M^{(Ab)_i}$ with both exponents at most $\ord(M)-1$, so $(Xb)_i=(Ab)_i$ for every $i$. Since the product of circulant matrices is commutative, $Xb=\Circ(b)\,y$ and $Ab=\Circ(b)\,a$, where $y=(y_0,\dots,y_{n-1})^T$ and $a=(a_0,\dots,a_{n-1})^T$. Hence $\Circ(b)(y-a)=0$, and since $\det\Circ(b)\ne0$, we obtain $y=a$, that is, $X=A$.
\end{proof}

\begin{corollary}\label{cor:unique}
Let $n\in\N$, let $v=(M^{b_0},\dots,M^{b_{n-1}})$ with $\Circ(b_0,\dots,b_{n-1})\in\Circ_n(\{1,\dots,k\})$ and $\det\Circ(b_0,\dots,b_{n-1})\ne0$, and let $A\in\Circ_n(\{0,\dots,t\})$ with $ntk\le\ord(M)-1$. Then $X=A$ is the only solution of $\Exp(X,v)=\Exp(A,v)$ with $X\in\Circ_n(\{0,\dots,t\})$.
\end{corollary}

\begin{proof}
For every $X\in\Circ_n(\{0,\dots,t\})$ and every $i$, we have $(Xb)_i\le ntk\le\ord(M)-1$, so Theorem \ref{thm:unique2} applies.
\end{proof}

We do not recommend using tuples satisfying the hypotheses of these results, since they consist of powers of a single matrix and hence reduce the problem to a discrete-logarithm problem. Nevertheless, they show that, under certain conditions, the private key is uniquely determined. Extending these results to more general tuples is an interesting problem for further research.

\subsection{Pohlig--Hellman-type attacks}

We first recall the following well-known result.

\begin{theorem}
Let $n\in\N$. The invertible elements of the multiplicative monoid $\Mat_n(\N)$ are the permutation matrices.
\end{theorem}

As a consequence, we obtain the following.

\begin{theorem}
Let $n\in\N$. The invertible elements of the multiplicative monoid $\Circ_n(\N)$ are the circulant permutation matrices.
\end{theorem}

As indicated in \cite{GZ}, Shanks' baby-step giant-step method and Pollard-rho-type attacks require a large number of invertible elements in order to be efficient. Since our semigroup has only $n$ invertible elements, these attacks are not efficient in our setting.

The Pohlig--Hellman reduction of \cite{GZ} exploits the finiteness of the acting semigroup: an element $s$ is chosen so that the problem on the smaller set $sS$ is easier to solve, and the solution is then lifted to the original problem. In our case, the semigroup $\Circ_n(\N)$ is infinite, and it becomes finite in practice only because the coefficients are bounded by $k$. If an attacker fixes $C\in\Circ_n([0,h])$ and attempts such a reduction, the resulting problem is posed on a set of circulant matrices whose entries are bounded by a multiple of the original bound, so the number of candidates is not effectively reduced. Hence we do not see how this reduction could give the attacker an advantage.

\subsection{The O--L attack}\label{sec:OL}

In \cite{OL}, the authors cryptanalysed the key-exchange protocol of \cite{MMR}. Since resistance to this attack is one of the main motivations for the new action, we describe the attack in detail, identify the structural properties it relies on, and explain why these properties fail for the exponentiation action.

In the protocol of \cite{MMR}, the public data are $M_1,M_2,X\in\Mat_m(R)$. Alice chooses $p_A,q_A\in\CR[x]$ and publishes $pk_A=p_A(M_1)Xq_A(M_2)$; Bob chooses $p_B,q_B\in\CR[x]$ and publishes $pk_B=p_B(M_1)Xq_B(M_2)$. The shared key is
\[
K=p_A(M_1)p_B(M_1)\,X\,q_B(M_2)q_A(M_2).
\]
Given the public data and a bound $h$ on the degrees, the attacker looks for a polynomial in three non-commuting variables of the form
\[
F(U,Y,Z)=\sum_{i,j=0}^{h}c_{ij}\,U^{i}YZ^{j},\qquad c_{ij}\in\CR,
\]
such that $F(M_1,X,M_2)=pk_A$. Written entrywise, this condition is a system of $m^2$ equations over $R$ in the $(h+1)^2$ unknowns $c_{ij}$, in which the unknowns appear \emph{linearly}. The authors of \cite{OL} show how to solve such systems over any finite congruence-simple semiring. Once any solution has been found, not necessarily equal to the pair $(p_A,q_A)$, the attacker computes
\begin{align*}
F(M_1,pk_B,M_2)&=\sum_{i,j}c_{ij}M_1^{i}\,p_B(M_1)Xq_B(M_2)\,M_2^{j}
=p_B(M_1)\Bigl(\sum_{i,j}c_{ij}M_1^{i}XM_2^{j}\Bigr)q_B(M_2)\\
&=p_B(M_1)\,pk_A\,q_B(M_2)=K,
\end{align*}
where we have used that the $c_{ij}$ are central, that $M_1^i$ commutes with $p_B(M_1)$ and $M_2^j$ with $q_B(M_2)$, and distributivity. More generally, $F(M_1,h(M_1)Xl(M_2),M_2)=h(M_1)\,pk_A\,l(M_2)$ for all $h,l\in\CR[x]$.

The argument relies on three properties.
\begin{enumerate}
\item[(L)] \emph{Linearity.} The candidate maps $Y\mapsto F(M_1,Y,M_2)$ form a family parametrised linearly by the coefficients $c_{ij}$, so that the condition $F(M_1,X,M_2)=pk_A$ is a system of linear equations over $R$.
\item[(C)] \emph{Compatibility.} Every map in this family commutes with every possible action of the other party: $F(M_1,\,p(M_1)Yq(M_2),\,M_2)=p(M_1)F(M_1,Y,M_2)q(M_2)$ for all $p,q\in\CR[x]$. Consequently, \emph{any} solution of the linear system, and not only the true private key, yields the shared key.
\item[(D)] \emph{Finite dimension.} The degree bound $h$ makes the family finite, with $(h+1)^2$ parameters, so the linear system can be solved in time polynomial in $h$ and $m$.
\end{enumerate}

However, for the exponentiation action, for $C\in\Circ_n(\N)$ let $\varphi_C(w)=\Exp(C,w)$. An attack in the spirit of \cite{OL} would look for a family $\mathcal F$ of maps $\psi$ on $\CR[M]^n$ such that (L) the condition $\psi(v)=pk_1$ is a linear system in the parameters of $\psi$, and (C) $\psi\circ\varphi_B=\varphi_B\circ\psi$ for every $B\in\Circ_n(\N)$, so that $\psi(pk_2)=\psi(\varphi_B(v))=\varphi_B(\psi(v))=\varphi_B(pk_1)=K$. We now show that the two natural candidate families each satisfy only one of these properties.

\emph{Monomial maps satisfy (C) but not (L).} The maps $\psi=\varphi_C$ with $C\in\Circ_n(\N)$ commute with all $\varphi_B$, since $\Circ_n(\N)$ is commutative. However, finding $C$ with $\varphi_C(v)=pk_1$ is exactly Problem \ref{prob:DL}: the unknowns are the exponents $c_0,\dots,c_{n-1}$, which enter the equations $\prod_j v_j^{c_{i-j}}=pk_{1,i}$ multiplicatively, as in a discrete-logarithm problem, and not linearly. The bound $k$ on the exponents plays the role of the degree bound in (D), but it only limits the size of an exhaustive search; it does not produce a linear system.

\emph{Polynomial maps satisfy (L) but not (C).} Following \cite{OL}, the attacker can linearise the problem by allowing polynomial expressions. For each $i$, one looks for a polynomial $P_i\in\CR[x_0,\dots,x_{n-1}]$ of bounded degree with $P_i(v)=pk_{1,i}$; since the monomials in $v$ are known matrices, this is a linear system in the coefficients of $P_i$, which can be solved by the methods of \cite{OL}. (In the same way, since $pk_{1,i}\in\CR[M]$, one can find $q_i\in\CR[x]$ with $q_i(M)=pk_{1,i}$.) The resulting map $\psi(w)=(P_i(w))_i$ satisfies $\psi(v)=pk_1$, but in general $(P_i(\varphi_B(w)))_i\ne\varphi_B(\psi(w))$, because $\varphi_B$ is built from products and powers and is not additive. The following example shows that this already fails in the simplest case.

\begin{proposition}\label{prop:OL}
Let $R$ be an additively idempotent semiring, $n=1$, and let $w\in\CR[M]$, so that $\varphi_b(w)=w^b$ for $b\in\N=\Circ_1(\N)$. Let $P(x)=x^s+x^t$ with $s<t$. Then
\[
P(w)^2=w^{2s}+w^{s+t}+w^{2t},\qquad P(w^2)=w^{2s}+w^{2t}.
\]
In particular, $P\circ\varphi_2=\varphi_2\circ P$ holds only if $w^{2s}+w^{s+t}+w^{2t}=w^{2s}+w^{2t}$, which fails in general.
\end{proposition}

\begin{proof}
Since $w$ commutes with itself, $(w^s+w^t)^2=w^{2s}+w^{s+t}+w^{t+s}+w^{2t}$, and $w^{s+t}+w^{s+t}=w^{s+t}$ by additive idempotence.
\end{proof}

Thus, if the linear system $P(w)=pk_1$ admits a solution $P$ with more than one term, which the linear algebra of \cite{OL} does not exclude, the attacker obtains $P(pk_2)=P(w^b)$, which in general differs from the key $pk_1^{\,b}=P(w)^b$. The only polynomial maps that commute with all $\varphi_b$ for every $w$ are the monomials $x^c$, and these return us to Problem \ref{prob:DL}. For $n>1$ the situation is analogous: a compatible solution must reproduce the circulant pattern of exponents in every component, and restricting the linear system to such monomial solutions amounts to an exhaustive search over the exponents.

The attack of \cite{OL} succeeds because, for the polynomial action, the family of maps compatible with the other party's action coincides with a family that is parametrised linearly. For the exponentiation action, these two families differ: the compatible maps are parametrised in the exponents, while the linearly parametrised maps are not compatible. We stress that this is not a proof of security; it shows that the specific mechanism of \cite{OL} does not carry over. The security of the protocol then rests on the hardness of Problem \ref{prob:DL}, which depends on the choice of the tuple $v$.

\subsection{Shor algorithm}\label{sec:quantum}

Shor \cite{Shor} introduced a quantum algorithm for computing discrete logarithms in abelian groups, and Childs and Ivanyos \cite{CI} generalised it to compute discrete logarithms in semigroups.

In our setting, if the commutative tuple consists of powers of a single matrix $M$, this algorithm yields a system of linear Diophantine equations in the entries of the private key, which can be solved to recover it. To prevent this attack, the components of the tuple should instead be general polynomials in $M$, rather than monomials $M^{b_j}$. In that case, we are not aware of any way to adapt the existing algorithm to compromise the security of the protocol. Moreover, knowledge of the order or the preperiod of $M$ does not appear to affect the security of the protocol.

\section{Conclusion}\label{sec:conclusion}
We have studied the action of circulant matrices over $\N$ on tuples of pairwise commuting matrices over a finite congruence-simple semiring, and the key-exchange protocol it induces. This action coincides with the one used by Huang and Li \cite{HL22} over the tropical semiring, and it is an instance of the framework of Maze, Monico and Rosenthal \cite{MMR02,MMR}. Our contribution is its study over finite congruence-simple semirings. We have described an explicit construction of matrices of large order, based on block-cyclic matrices and conjugation by generalised permutation matrices, and we have given a sufficient condition for the additive irreducibility of a finite semigroup, which shows that these conjugations are exactly those given by the invertible matrices.

We have analysed the computational cost of the protocol and compared it with the protocol of \cite{MMR} and with the tropical scheme of \cite{HL22}. Since the action is the same, the computational cost coincides with that of \cite{HL22} for equal parameters, while the fixed size of the entries over a finite semiring leads to smaller public keys. Moreover, the dimension required for a given range of exponents grows only poly logarithmically, instead of linearly as in the tropical setting. We have also given conditions under which the private key is uniquely determined, and we have examined known attacks. In particular, we have shown that the structural properties exploited by the attack of \cite{OL} on the polynomial action of \cite{MMR} do not hold for the exponentiation action.

Several questions remain open. The most important are construction of tuples in $\CR[M]$ whose components have large order and are not powers of a single matrix, a more precise analysis of the hardness of Problem \ref{prob:DL} for such tuples, and extensions of the uniqueness results of Section \ref{sec:brute} to general tuples. Further questions include the study of the protocol over different semirings, as well as over different commutative sets that do not rely on matrices over a semiring.

\section*{Acknowledgements}
We would like to express our sincere gratitude to Jens Zumbragel for his valuable feedback and insightful comments on earlier versions of this work. We are also deeply indebted to Sergey Sergeev for introducing us to the work of Huang and Li, which significantly shaped the direction of this research. We also acknowledge the use of artificial intelligence tools to improve the English language and presentation of this manuscript.

This research was supported by the Spanish Ministry of Science, Innovation, and Universities under the FPU 2023 grant programme, and the project PID2023-149203NB-I00.

\end{document}